\documentclass[10pt]{amsart}
\usepackage[margin=1.2in]{geometry}
\usepackage{color}
\usepackage{refcheck}
\newtheorem{theorem}{Theorem}[section]
\newtheorem{lemma}[theorem]{Lemma}
\newtheorem{proposition}{Proposition}[section]

\theoremstyle{definition}
\newtheorem{definition}[theorem]{Definition}

\theoremstyle{remark}
\newtheorem{remark}[theorem]{Remark}

\numberwithin{equation}{section}

\newcommand{\blankbox}[2]{%
	\parbox{\columnwidth}{\centering
		\setlength{\fboxsep}{0pt}%
		\fbox{\raisebox{0pt}[#2]{\hspace{#1}}}%
\makeatother
\makeatletter
\@namedef{subjclassname@2020}{\textup{2020} Mathematics Subject Classification}
\makeatother
	}%
}

\usepackage{graphicx}
\DeclareGraphicsExtensions{.pdf,.png,.jpg}
\begin{document}
	\title[Relativistic BGK model for gas mixtures]{Stationary solutions to the relativistic BGK model for gas mixtures in a slab}

	\author[B.-H. Hwang]{Byung-Hoon Hwang}
	\address{Department of Mathematics Education, Sangmyung University, 20 Hongjimun 2-gil, Jongno-Gu, Seoul 03016, Republic of Korea}
	\email{bhhwang@smu.ac.kr}

	\author[M.-S. Lee ]{Myeong-Su Lee}
	\address{Department of Mathematical Science, Korean Advanced Institute of Science and Technology, Daehak-ro 291, Yuseong-gu, Daejeon, Republic of Korea	}
	\email{msl3573@kaist.ac.kr}


	
\begin{abstract}
In a recent paper \cite{HLY2},  the authors proposed a BGK model for relativistic gas mixtures  based on the Marle-type approximation, which satisfies the fundamental kinetic properties: non-negativity of distribution functions, conservation laws, H-theorem, and indifferentiability principle. In this paper, we are concerned with the stationary problems to the relativistic BGK model for gas mixtures in slab geometry.  We establish the existence of a unique mild solution with the fixed inflow boundary data when the collision frequencies for each species are sufficiently small.  
 \end{abstract}
	\maketitle

\section{Introduction}
\subsection{History and Motivation}
 The Boltzmann equation is fundamental in kinetic theory. However, its practical application, particularly in simulating various flow problems, is limited due to its high computational costs. For this reason, the BGK model \cite{BGK,W} was proposed as a relaxation-time approximation of the Boltzmann equation, replacing the complicated Boltzmann collision operator with a simpler one (so-called the BGK operator). Due to the consistency with the Boltzmann equation as well as its efficiency, many researchers have developed the BGK approximation to model the gas dynamics in various physical contexts. In the relativistic framework, the BGK model was generalized by Marle \cite{Marle,Marle2}, and by Anderson and Witting \cite{AW} for a monatomic single gas.  In \cite{PR}, Pennisi and Ruggeri suggested a new relativistic BGK model for a polyatomic single gas by applying the  generalized relativistic Maxwellian (often called the J\"{u}ttner distribution), see \cite{PR2}. These models have been fruitfully used to understand various relativistic flow problems \cite{CKT,DHMNS,Kremer2,MBHS,MNR}. 
 
 In the past decade, there have been extensive mathematical studies on the above relativistic BGK models. The earliest mathematical research can be traced back to \cite{BCN} where the unique determination of equilibrium parameters, scaling limits, and linearized problems were addressed for the Marle model. In \cite{BNU}, the global existence and asymptotic stability for the nonlinear Marle model were established in the near-equilibrium regime. The global existence of weak solutions was studied in \cite{CJS} by developing the averaging lemma in the relativistic context. Stationary problems were considered in \cite{HY3} for fixed inflow boundary data in a slab. Regarding the Anderson–Witting model, the global existence and large-time behavior of classical solutions were established in \cite{HY2}. The unique determination of equilibrium parameters and the existence of bounded solutions were discussed in \cite{Hwang}.  It is found in \cite{HLY} that the Anderson–Witting model is solvable explicitly in FLRW spacetime. The stationary problems were studied in  \cite{HY4}. For the Pennisi–Ruggeri model, the global existence and large-time behavior were studied in \cite{HRY}.

From a practical point of view, it is important to consider systems of gas mixtures as well as single-species gases.  Despite this importance, in contrast to the extensive studies of relativistic BGK models for single-species gases, the literature for relativistic BGK models for gas mixtures has been extremely limited. The research on relativistic BGK models for gas mixtures had usually been focused on the formal computation of the transport coefficients without providing a complete presentation of the relaxation operators (see \cite{Kremer,Kremer2,Kremer3}), and thus the existence theory could not be studied yet. Motivated by this, the authors \cite{HLY2} recently proposed the relativistic BGK model for gas mixtures with a precise definition of relaxation operators.  In this paper, we  aim to establish the existence and uniqueness of stationary solutions to the recent relativistic BGK model \cite{HLY2} in slab geometry.

\subsection{Relativistic BGK model for gas mixtures}  Let us consider relativistic gas mixtures of $N$ constituents in Minkowski space-time with metric tensor $\eta_{\mu\nu}$ and its inverse $\eta^{\mu\nu}$: 
$$
\eta_{\mu\nu}=\eta^{\mu\nu}=\text{diag}(1,-1,-1,-1).
$$
where Greek indices range from $0$ to $3$. We use the raising and lowering indices
$$
\eta_{\mu\nu} a^\nu= a_\mu,\qquad \eta^{\mu\nu} a_\nu= a^\mu
$$ and follow the Einstein summation convention so that
$$
a^\mu b_\mu =\eta_{\mu\nu}a^\mu b^\nu= a^0 b^0-\sum_{j=1}^3a^j  b^j=	a_\mu b^\mu.
$$
The particles of $i$  species $(i=1,\cdots, N)$ are characterized by the space-time coordinates $x^{\mu}$ and the four-momentums $p_i^{\mu}$
$$
x^{\mu} =(ct,x)\in \mathbb{R}_+\times \Omega,\qquad p_i^{\mu}=(\sqrt{(cm_i)^2+|p_i|^2},p_i)\in\mathbb{R}_+\times\mathbb{R}^3.
$$ 
 where $c$ is the speed of light and $m_i$ is the rest mass of $i$ species. Let $f_i(x^{\mu}, p_i^{\mu})$ be the momentum distribution function representing the number density of $i$ species  at the phase point $(x^{\mu},p_i^{\mu})$. Then the relativistic BGK model proposed in \cite{HLY2} reads
 \begin{align}\label{RBGK}
\partial_t f_i+\frac{cp_i}{p_i^0}\cdot\nabla_x f_i=\frac{cm_i}{\tau_i p_i^0}(\mathcal{J}_i-f_i)=:Q_i, \qquad i=1,\cdots,N.
\end{align} 
Here $\tau_i$ denotes the characteristic time of order of the time between collisions and $\mathcal{J}_i$ is the attractor designed based on the J\"{u}ttner distribution \cite{Juttner}, given by
\begin{equation}\label{Juttner2}
\mathcal{J}_i=\frac{\int_{\mathbb{R}^3} f_i \frac{dp_i}{p_i^0}}{\int_{\mathbb{R}^3} e^{-c\widetilde{\beta} p_i^0} \frac{dp_i}{p_i^0}}e^{-\widetilde{\beta} \widetilde{U}^\mu p_{i \mu }},\quad\mbox{with} \quad \widetilde{\beta}:=\frac{1}{k\widetilde{T}} 
\end{equation}
where $k$ is the Boltzmann constant, $\widetilde{U}^\mu$ is the auxiliary four-velocity, and $\widetilde{T}=1/k\widetilde{\beta}$ is the auxiliary temperature. To define $\widetilde{U}^\mu$ and $\widetilde{\beta}$, we introduce the partial particle four-flow $N_i^\mu$ and the partial energy-momentum tensor $T_i^{\mu\nu}$:
$$
N_i^\mu=c\int_{\mathbb{R}^3}p^\mu_if_i \,\frac{dp_i}{p_i^0},\qquad T_i^{\mu\nu}=c\int_{\mathbb{R}^3}p_i^\mu p_i^\nu f_i \frac{dp_i}{p_i^0}.
$$
According to the Eckart frame \cite{Eckart}, $N_i^\mu$ is decomposed into
\begin{equation*}
N_i^\mu= n_iU_i^\mu
\end{equation*}
where $n_i$ is the (macroscopic) number density and $U_i^\mu$ is the Eckart four-velocity defined by
\begin{align*}\begin{split}
n_i&=\frac{1}{c}\int_{\mathbb{R}^3}p^\mu_i U_{i \mu }f_i \,\frac{dp_i}{p_i^0}=\Biggl\{\biggl(\int_{\mathbb{R}^3}f_i \,dp_i\biggl)^2-\sum_{j=1}^3\biggl(\int_{\mathbb{R}^3}p^j_if_i \,\frac{dp_i}{p_i^0}\biggl)^2\Biggl\}^{\frac{1}{2}},\cr
U_i^\mu&=\frac{c}{n_i}\int_{\mathbb{R}^3}p^\mu_if_i \,\frac{dp_i}{p_i^0}
\end{split}\end{align*}
respectively. Here $U^\mu_i$ has a constant length in the following sense
$$
U^\mu_i U_{i \mu }=c^2,\qquad \text{and hence}\qquad U_i^\mu=\left(\sqrt{c^2+|U_i|^2},U_i\right).
$$
Then the auxiliary four-velocity $\widetilde{U}^\mu$ is defined by
\begin{align}\label{def_U}
\widetilde{U}^\mu=c\frac{\sum_{i=1}^{N}\frac{m_i}{\tau_i}n_iU_i^\mu}{\sqrt{\left( \sum_{i=1}^{N}\frac{m_i }{\tau_i}n_iU_i^\mu\right)\left( \sum_{j=1}^{N}\frac{m_j }{\tau_j}n_jU_{j\mu}\right)}}, 
\end{align}
and thus $\widetilde{U}^\mu \widetilde{U}_\mu=c^2$. The auxiliary parameter $\widetilde{\beta}$ is determined by the nonlinear relation
\begin{align}\label{beta identity}
\sum_{i=1}^{N}\frac{m_i}{\tau_i}\Phi_i(\widetilde{\beta})\int_{\mathbb{R}^3}f_i \,\frac{dp_i}{p_i^0}=\frac{1}{c}\left[\left( \sum_{i=1}^{N}\frac{m_i }{\tau_i}n_iU_i^\mu\right)\left( \sum_{j=1}^{N}\frac{m_j }{\tau_j}n_jU_{j\mu}\right)\right]^{\frac{1}{2}}.
\end{align}
where $\Phi_i$ denotes
$$
\Phi_i(\widetilde{\beta})=\frac{\int_{\mathbb{R}^3} e^{-c\widetilde{\beta}p_i^0}\,dp_i}{\int_{\mathbb{R}^3} e^{-c\widetilde{\beta}p_i^0}\,\frac{dp_i}{p_i^0}}.
$$
Note from \cite{HLY} that if $f_i$ $(i=1,\cdots,N)$ is non-negative, not trivially zero, and belongs to $L^1(\mathbb{R}^3)$, then $\widetilde{\beta}$ is uniquely determined by the relation \eqref{beta identity}. Due to \eqref{def_U} and \eqref{beta identity}, the attractor $\mathcal{J}_i$  satisfies
\begin{align*}
\begin{split}
&\int_{\mathbb{R}^3} Q_i\, dp_i=0 \quad (i=1,\cdots,N), \qquad \sum_{i=1}^{N}\int_{\mathbb{R}^3} p_i^\mu Q_i\,dp_i=0,
\end{split}
\end{align*}
which leads to the conservation laws for the partial particle four-flows $N_i^\mu$ and the energy-momentum tensor $T^{\mu\nu}$ of gas mixtures:
$$
\frac{\partial N_i^\mu}{\partial x^\mu}=0,\qquad \frac{\partial }{\partial x^\nu}\sum_{i=1}^{N}T_i^{\mu\nu}=0.
$$
Also, the standard argument gives the H-theorem:
	$$
\frac{\partial S^\mu}{\partial x^\mu} \ge 0.
$$
Here $S^\mu$ denotes the entropy four-flow of gas mixtures defined by
$$
S^\mu= -kc\sum_{i=1}^N\int_{\mathbb{R}^3}p_i^\mu  f_i \ln\left( \frac{f_i h^3}{g_{s_i}}\right) \,\frac{dp_i}{p_i^0}.
$$
where $h$ is the Planck constant and $g_{s_i}$ the degeneracy factor of $i$ species with spin $s_i$:
$$
g_{s_i}=\begin{cases}
2s_i+1 & \text{for}\quad m\neq 0\cr
2s_i & \text{for}\quad m=0
\end{cases}.
$$

\subsection{Slab problems and main result} In this subsection, we briefly explain the slab problems and then state the main result of this paper. The slab problems describe the situation that gas molecules are assumed to be between two parallel planes. Simply put, these planes are positioned at $x=0$ and $x=1$, and are orthogonal to $x-$axis, and the distribution of gas molecules depend only on the direction of $x$. Such problems can be applied in the analysis of the Knudsen layer \cite{SAY86}.  For more physical examples, we refer to \cite{ATK,cer} and references therein.

The stationary problem of the relativistic BGK model \eqref{RBGK} in slab geometry is written as:
\begin{align}\label{SRBGK}
	p_i^1 \partial_xf_i=\omega_i(\mathcal{J}_i-f_i), \qquad (x,p_i)\in [0,1]\times \mathbb{R}^3
\end{align}
for $i=1,\cdots,N$, subject to the fixed inflow boundary data 
$$
f_i(0,p)=f_{i,L}(p_i)\ \ \text{for}\ \ p_i^1>0,\quad f_i(1,p)=f_{i,R}(p_i)\ \ \text{for}\ \ p_i^1<0,
$$
where $\omega_i:=\frac{m_i}{\tau_i} $ stands for the collision frequency of species $i$. Throughout this paper, the boundary data is often denoted as
$$f_{i,LR} =f_{i,L}\bold{1}_{p_i^1>0}+f_{i,R}\bold{1}_{p_i^1<0}$$
for brevity. Define the notations:
\begin{align}
	\begin{split}\label{lambda}
		&a_{i,\ell}=\frac{1}{4}\int_{\mathbb{R}^3} f_{i,LR}\frac{dp_i}{p_i^0} ,\qquad a_{i,u}=2\int_{\mathbb{R}^3} f_{i,LR}\, dp_i,\\
		&\lambda_i =\left(\int_{\mathbb{R}^3} f_{i,LR} \,dp_i\int_{\mathbb{R}^3}\frac{1}{(p_i^0)^2} f_{i,LR} \,dp_i\right)^{\frac{1}{2}}\left(\int_{\mathbb{R}^3} f_{i,LR} \, \frac{dp_i}{p_i^0}\right)^{-1},\\
		&\gamma=\underset{i=1,\cdots, N}{\max}\{\lambda_i^{-1/4}\}.
	\end{split}
\end{align}
\begin{remark}\label{lambda1}
	The Cauchy-Schwartz inequality, together with the fact that $f_{i,LR}$ and $f_{i,LR}/(p_i^0)^2$ are not identically equal, implies that the constants $\lambda_i$ are strictly greater than $1$. 
\end{remark}
We introduce our concept of solutions to the system of \eqref{SRBGK}.
\begin{definition}
	A $N$-tuple of functions $f:= \left(f_1(x,p_1),\cdots,f_N(x,p_N) \right)$ consisting of non-negative $L^1$-functions is called a mild solution to the system \eqref{SRBGK} if each $f_i$ $(i=1,\cdots,N)$ is given by
	\begin{align*}
		f_i(x,p_i)&=\left(e^{-\frac{\omega_i}{|p_i^1|}x}f_{i,L}+\frac{\omega_i}{|p_i^1|}\int_0^xe^{-\frac{\omega_i}{|p_i^1|}(x-y)}\mathcal{J}_i(y)\, dy\right)\bold{1}_{p_i^1>0}\cr
		&+\left(e^{-\frac{\omega_i}{|p_i^1|}(1-x)}f_{i,R}+\frac{\omega_i}{|p_i^1|}\int_x^1e^{-\frac{\omega_i}{|p_i^1|}(y-x)}\mathcal{J}_i(y)\,dy\right)\bold{1}_{p_i^1<0}.
	\end{align*}
\end{definition} 

\begin{align}\label{lambda}
	\lambda_i =\left(\int_{\mathbb{R}^3} f_{i,LR} \,dp_i\int_{\mathbb{R}^3}\frac{1}{(p_i^0)^2} f_{i,LR} \,dp_i\right)^{\frac{1}{2}}\left(\int_{\mathbb{R}^3} f_{i,LR} \, \frac{dp_i}{p_i^0}\right)^{-1},\quad \gamma=\underset{i=1,\cdots, N}{\max}\{\lambda_i^{-1/4}\}.
\end{align}
Then our main result is stated as follows.
\begin{theorem}\label{main}
	Suppose that the inflow boundary data satisfies
	$$
	f_{i,LR}\ge 0,\quad f_{i,LR}\in L^1(\mathbb{R}^3),\quad a_{i,\ell}>0.
	$$
	Then, there exists a constant $\varepsilon>0$ such that if $\omega_i<\varepsilon$ for all $i=1,\cdots,N$, \eqref{SRBGK} admits a unique mild solution $f=(f_1,\cdots,f_N)$ to the system \eqref{SRBGK} satisfying
	\begin{enumerate}
		\item Each $f_i$  satisfies 
		$$ f_i\ge 0,\qquad a_{i,\ell}\le \int_{\mathbb{R}^3} f_i\, \frac{dp_i}{p_i^0},\qquad \int_{\mathbb{R}^3} f_i \, dp_i\le a_{i,u}.$$
		\item The auxiliary parameter $\tilde{\beta}$ determined by \eqref{beta identity} satisfies
		$$  \gamma\sum_{i=1}^{N}cm_i\omega_i \lambda_i  \int_{\mathbb{R}^3} f_{i,LR}\,\frac{dp_i}{p_i^0}\le \sum_{i=1}^{N}\omega_i \Phi_i(\widetilde{\beta}) \sqrt{\lambda_i} \int_{\mathbb{R}^3} f_{i,LR}\,\frac{dp_i}{p_i^0}.
		$$
	\end{enumerate}
	
	%for some constant $\lambda_w$ to be determined in Lemma \ref{222}
\end{theorem}

The proof of Theorem \ref{main} relies on the Banach fixed point approach developed in \cite{BY}, so the main ingredient of the proof is to construct the solution space $\Omega$ appropriately. For a mild form of \eqref{SRBGK} to be a proper solution operator, i.e. contractive and it maps $\Omega$ into $\Omega$, it is necessary to adopt some suitable solution space to control the attractor  $\mathcal{J}_i$ of \eqref{Juttner2}. The difficulty of this work arises here since the auxiliary parameters $\widetilde{\beta}$ is defined by a one-to-one correspondence via the highly nonlinear relation \eqref{beta identity}:
$$
\sum_{i=1}^{N}\frac{m_i}{\tau_i}\Phi_i(\widetilde{\beta})\int_{\mathbb{R}^3}f_i \,\frac{dp_i}{p_i^0}=\frac{1}{c}\left[\left( \sum_{i=1}^{N}\frac{m_i }{\tau_i}n_iU_i^\mu\right)\left( \sum_{j=1}^{N}\frac{m_j }{\tau_j}n_jU_{j\mu}\right)\right]^{\frac{1}{2}},
$$
see \cite[Proposition 4.2]{HLY} for details. To obtain a bound of $\widetilde{\beta}$, a careful and delicate constraint is required on the solution space since $\Phi_i$ is entangled with the sum of moments of $f_i$ $(i=1,\cdots,N)$ and the constraint must be closed in our solution map. In fact, the trivial lower bound of $\Phi_i(\widetilde{\beta})$, that is $cm_i$, do not give any information on $\widetilde{\beta}$ since $\Phi_i^{-1}(y)$ goes to $\infty$ as $y$ approaches $cm_i$. To overcome this difficulty, we impose the constraint on the solution space using $\lambda_i$ given in \eqref{lambda}, which enables us to close our argument. Details can be found in Lemma \ref{11} and Lemma \ref{222}.

 We also briefly review the relevant results on the stationary problems of the BGK-type models. For the original BGK model, the existence of stationary solutions of the slab problem was established in \cite{Ukai} where the authors used a Schauder type fixed point theorem. In \cite{Nouri}, the existence of weak solutions for the stationary quantum BGK model in a slab was studied with a discretized condensation term. The authors of \cite{BY} obtained the existence of the unique mild solution to the stationary problem in a slab for ES-BGK model by using the Banach fixed point theorem. This argument was widely adapted to various other BGK-type models. We refer to \cite{HY3,HY4} for the relativistic single-species gases, \cite{BaeY} for the quantum gases, and \cite{KLY} for chemically reacting gases.

\subsection{Outline} The paper is organized as follows. In Section 2, we introduce the solution space and several estimates for the Banach fixed-point framework. In Section 3, we present our solution operator which maps the solution space into itself. Finally, Section 4 is devoted to completing the proof of Theorem \ref{main} by showing that the solution operator is contractive under our assumptions.

\section{Solution space and preliminary estimates}

Let us define the solution space $\Omega$  by
$$
\left\{f= \left(f_1(x,p_1),\cdots,f_N(x,p_N) \right)\in  \prod_{i=1}^N L^1([0,1]\times \mathbb{R}^3)~ \Big|~ f~ \text{satisfies}~ (\mathcal{A})\ \mbox{and}\ (\mathcal{B}) \right\}
$$
in which the properties $(\mathcal{A})$ and $(\mathcal{B})$ are given as follows:
\begin{enumerate}
	\item[$(\mathcal{A})$] Each $f_i$  satisfies 
	$$ f_i\ge 0,\qquad a_{i,\ell}\le \int_{\mathbb{R}^3} f_i\, \frac{dp_i}{p_i^0},\qquad \int_{\mathbb{R}^3} f_i \, dp_i\le a_{i,u}.$$
	\item[$(\mathcal{B})$] Let $\widetilde\beta$ be determined uniquely by $f$ in the relation \eqref{beta identity}. Then $\widetilde{\beta}$ satisfies
	$$  \gamma\sum_{i=1}^{N}cm_i\omega_i \lambda_i  \int_{\mathbb{R}^3} f_{i,LR}\,\frac{dp_i}{p_i^0}\le \sum_{i=1}^{N}\omega_i \Phi_i(\widetilde{\beta}) \sqrt{\lambda_i} \int_{\mathbb{R}^3} f_{i,LR}\,\frac{dp_i}{p_i^0}.
	$$
\end{enumerate}
Here  $\Omega$ is a closed subspace of the Banach space $\prod_{i=1}^N L^1([0,1]\times \mathbb{R}_{p_i}^3) $. Throughout this paper, we choice the norm as a sum of the usual $L^1$-norms:
$$
\|f\|:= \sum_{i=1}^N \|f_i\|_{L^1}.
$$
 In the rest of this section, we present preliminary estimates to set up the fixed-point argument.

\begin{lemma}\label{11}
Let $f=(f_1,\cdots,f_N)$ be an element of the solution space $\Omega$. Then $\widetilde{U}$ and $\widetilde{\beta}$, defined by $f$ in the relations \eqref{def_U} and \eqref{beta identity}, are bounded  as 
\begin{align*}
|\widetilde{U}|\le \frac{\underset{i=1,\cdots,N}{\max}{a_{i,u}}}{\underset{i=1,\cdots,N}{\min}{m_ia_{i,\ell}}},\qquad
\underset{i=1,\cdots,N}{\min}\Phi_i^{-1}(2a_{i,u}/a_{i,\ell})\leq\widetilde{\beta}\leq \underset{i=1,\cdots,N}{\max}\Phi_i^{-1}(\gamma cm_i \sqrt{\lambda_i}).
\end{align*}
For simplicity, we denote these three bounds by $\widetilde{U}_u,\widetilde{\beta}_\ell$ and $\widetilde{\beta}_u$, respectively.
\end{lemma}
\begin{proof}
$\bullet$ (Estimate of $\widetilde U$): By definition, 
\begin{align*}
\left( \sum_{i=1}^{N}\omega_i n_iU_i^\mu\right)\left( \sum_{j=1}^{N}\omega_jn_jU_{j\mu}\right)&=c^2\int_{\mathbb{R}^3}\int_{\mathbb{R}^3} \left(\sum_{i=1}^{N}\omega_i p^\mu_if_i\right)\left(\sum_{j=1}^{N}\omega_j p_{j\mu}f_j\right) \,\frac{dp_i}{p_i^0}\,\frac{dp_j}{p_j^0}.
\end{align*}
Applying the Cauchy-Schwarz inequality, we see that
$$
p_i^\mu p_{j\mu}=\sqrt{(m_ic)^2+|p_i|^2}\sqrt{(m_jc)^2+|p_j|^2}-p_i\cdot p_j\ge m_im_jc^2,
$$
which gives
\begin{align}\label{lowerbound}\begin{split}
\left( \sum_{i=1}^{N}\omega_i n_iU_i^\mu\right)\left( \sum_{j=1}^{N}\omega_jn_jU_{j\mu}\right)&\ge c^2\int_{\mathbb{R}^3}\int_{\mathbb{R}^3} \left(\sum_{i=1}^{N}\omega_i m_if_i\right)\left(\sum_{j=1}^{N}\omega_j m_jf_j\right) \,\frac{dp_i}{p_i^0}\,\frac{dp_j}{p_j^0}\cr
&=\left(c\sum_{i=1}^{N}\omega_i m_i \int_{\mathbb{R}^3}  f_i  \,\frac{dp_i}{p_i^0} \right)^2.
\end{split}\end{align}
Thus, we have from $(\mathcal{A})$ that 
\begin{align*}
|\widetilde{U}| &= c\frac{\left|\sum_{i=1}^{N}\omega_in_iU_i\right|}{\sqrt{\left( \sum_{i=1}^{N}\omega_in_iU_i^\mu\right)\left( \sum_{j=1}^{N}\omega_jn_jU_{j\mu}\right)}}\cr
&\le \frac{\sum_{i=1}^{N}\omega_i \int_{\mathbb{R}^3}\frac{|p_i|}{p_i^0}f_i \, dp_i}{\sum_{i=1}^{N}\omega_i m_i \int_{\mathbb{R}^3}  f_i  \,\frac{dp_i}{p_i^0}}\cr
&\le \frac{\underset{i=1,\cdots,N}{\max}{a_{i,u}}}{\underset{i=1,\cdots,N}{\min}{m_ia_{i,\ell}}}
\end{align*}
$\bullet$ (Estimate of $\widetilde \beta$): By $(\mathcal{B})$, we know the following inequality holds:
\begin{equation*}
\gamma	cm_i\lambda_i  \le \Phi_i (\widetilde{\beta} ) \sqrt{\lambda_i}
\end{equation*}
for at least one $i\in \{1,\cdots, N\} $. We recall from   \cite[Appendix]{BCNS} that  $\Phi_i$ is strictly decreasing on $\widetilde{\beta}\in(0,\infty)$ and its range is $(cm_i,\infty)$. Here, due to H\"{o}lder's inequality, $\lambda_i$ is strictly bigger than $1$  and hence the following holds:
$$
\frac{1}{\sqrt{\lambda_i}}\gamma cm_i\lambda_i\geq cm_i\lambda_i^{1/4}>cm_i.
$$
This gives that $\Phi_i^{-1}(\gamma cm_i \sqrt{\lambda_i})$ is well-defined and $\widetilde{\beta}$ is bounded above as
$$
\widetilde{\beta}\leq\Phi_i^{-1}(\gamma cm_i\sqrt{\lambda_i}).
$$
Taking maximum over $i$, we obtain
$$
\widetilde{\beta}\leq \widetilde{\beta}_u:=\underset{i=1,\cdots,N}{\max}\Phi_i^{-1}(\gamma cm_i\sqrt{\lambda_i}).
$$
For the lower bound of $\widetilde{\beta}$, we observe that
\begin{align}\label{upperbound}
\begin{split}
&\left(\sum_{i=1}^{N}\omega_i n_iU_i^\mu \right)\left(\sum_{j=1}^{N}\omega_j n_jU_{j\mu}\right)\cr
&=c^2\sum_{i=1}^{N}\sum_{j=1}^{N}\omega_i\omega_j\left( \int_{\mathbb{R}^3}p^0_if_i \,\frac{dp_i}{p_i^0}\int_{\mathbb{R}^3}p^0_jf_j \,\frac{dp_j}{p_j^0} -\sum_{k=1}^3\int_{\mathbb{R}^3}p^k_if_i \,\frac{dp_i}{p_i^0}\cdot \int_{\mathbb{R}^3}p^k_jf_j \,\frac{dp_j}{p_j^0}\right)\cr
&\le 4c^2\sum_{i=1}^{N}\sum_{j=1}^{N}\omega_i\omega_j\int_{\mathbb{R}^3}p^0_if_i \,\frac{dp_i}{p_i^0}\int_{\mathbb{R}^3}p^0_jf_j \,\frac{dp_j}{p_j^0}\cr
&=\left(2c\sum_{i=1}^{N}\omega_i \int_{\mathbb{R}^3}f_i \,dp_i\right)^2.
\end{split}
\end{align}
This, together with \eqref{beta identity}, gives that
\begin{align*}
\sum_{i=1}^{N}\omega_i\Phi_i(\widetilde{\beta}) a_{i,\ell}\le	\sum_{i=1}^{N}\omega_i\Phi_i(\widetilde{\beta})\int_{\mathbb{R}^3}f_i \,\frac{dp_i}{p_i^0}&=\frac{1}{c}\left[\left( \sum_{i=1}^{N}\omega_in_iU_i^\mu\right)\left( \sum_{j=1}^{N}\omega_jn_jU_{j\mu}\right)\right]^{\frac{1}{2}}		\le 2\sum_{i=1}^{N}\omega_i a_{i,u}. 
\end{align*}
This implies that for at least one $j\in\{1,\cdots, N\}$, the following holds:
\begin{align}\label{beta_l}
	\Phi_j(\widetilde{\beta})\leq\frac{2a_{j,u}}{a_{j,\ell}}.	
\end{align}
Since the right hand side of \eqref{beta_l} is bigger than $16cm_i$ by definition, we can take the inverse $\Phi^{-1}$ on both sides to obtain 
\begin{align*}
\min_{j=1,\cdots, N}	\Phi_j^{-1}\left(\frac{2a_{j,u}}{a_{j,\ell}}\right)=:\widetilde{\beta}_\ell\leq\widetilde{\beta},
\end{align*}
which completes the proof.
\end{proof}	
\begin{lemma}\label{22}
Let $f=(f_1,\cdots,f_N)$ be an element of the solution space $\Omega$.  Then, there exist positive constants $C_{1}$ and $C_2$ depending on the quantities given in \eqref{lambda} such that 
$$
\mathcal{J}_i   \le C_1e^{-C_2p_i^0}\qquad \text{for all $i=1,\cdots,N$},
$$
where $\mathcal{J}_i$ is given as \eqref{Juttner2}.
\end{lemma}
\begin{proof} It follows from Lemma \ref{11} that 
$$
0\le \frac{\int_{\mathbb{R}^3} f_i \frac{dp_i}{p_i^0}}{\int_{\mathbb{R}^3} e^{-c\widetilde{\beta} p_i^0} \frac{dp_i}{p_i^0}} \le \frac{a_{i,u}}{\int_{\mathbb{R}^3} e^{-c\widetilde{\beta}_u p_i^0} \frac{dp_i}{p_i^0}}=:C_1.
$$
Also, one can obtain
\begin{align*}\begin{split}
		\widetilde{\beta}\widetilde{U}^\mu p_{i \mu }&=\widetilde{\beta}\left(\sqrt{c^2+|\widetilde{U}|^2}p_i^0-\widetilde{U}\cdot p_i\right)\cr
		&\ge \widetilde{\beta}\left(\sqrt{c^2+|\widetilde{U}|^2}-|\widetilde{U}|\right)p_i^0\cr
		&\ge \widetilde{\beta}_\ell\left(\sqrt{c^2+\widetilde{U}_u^2}-\widetilde{U}_u\right)p_i^0\cr
		&=:C_2p_i^0,
	\end{split}
\end{align*}
where we used the fact that $\sqrt{c^2+x^2}-x$ is decreasing on $x\in [0,\infty)$.
Therefore we conclude that
$$
\mathcal{J}_i=\frac{\int_{\mathbb{R}^3} f_i \frac{dp_i}{p_i^0}}{\int_{\mathbb{R}^3} e^{-c\widetilde{\beta} p_i^0} \frac{dp_i}{p_i^0}}e^{-\widetilde{\beta}\widetilde{U}^\mu p_{i \mu }} \le  C_1e^{-C_2p_i^0}.
$$
\end{proof}
\begin{lemma}\label{phi2}
Let $f=(f_1,\cdots,f_N)$ be an element of the solution space $\Omega$. Assume that $\omega_i\in (0,1)$ for every $i\in\{1,\cdots,N\}$. We then have
\begin{align*}
& \int_{p_i^1>0}\frac{\omega_i}{|p_i^1|}\int_0^xe^{-\frac{\omega_i}{|p_i^1|}(x-y)}\mathcal{J}_i(y)\, dydp_i+\int_{p_i^1<0}\frac{\omega_i}{|p_i^1|}\int_x^1e^{-\frac{\omega_i}{|p_i^1|}(y-x)}\mathcal{J}_i(y)\,dydp_i\cr
&\le C\left(\omega_i\ln\frac{1}{\omega_i}+\omega_i+\omega_i^2e^{-\frac{C}{\omega_i}}\right).
\end{align*}
\end{lemma}
\begin{proof} Due to the symmetry, it  is sufficient to prove only the case $p_i^1>0$. It is straightforward that
\begin{align*}
	\int_{\mathbb{R}^2} e^{-C_2\sqrt{(cm_i)^2+|p_i|^2}}\,dp_i^2dp_i^3&\le \int_{\mathbb{R}^2} e^{-C_3(|p_i^1|+|p_i^2|+|p_i^3|)}\,dp_i^2 dp_i^3 \cr
	&=e^{-C_3|p_i^1|}\int_{\mathbb{R}^2} e^{-C_3(|p_i^2|+|p_i^3|)}\,dp_i^2dp_i^3\cr
	&=\frac{4}{C_3^2}e^{-C_3|p_i^1|}\end{align*}
where we used the elementary inequality in the first line,
\begin{align*}
	\sqrt{C+|x|^2+|y|^2+|z|^2}\geq \frac{|x|+|y|+|z|}{\sqrt{3}},
\end{align*}	 
and we denote $C_3$ as $C_2/\sqrt{3}$.
This, together with Lemma \ref{22}, gives that
\begin{align}\label{p1>0}
\begin{split}
\int_{p_i^1>0}\frac{\omega_i}{p_i^1}\int_0^xe^{-\frac{\omega_i}{p_i^1}(x-y)}\mathcal{J}_i(y)\,dydp_i&\le C_1\int_{p_i^1>0}\frac{\omega_i}{p_i^1}\int_0^xe^{-\frac{\omega_i}{p_i^1}(x-y)}e^{-C_2\sqrt{(cm_i)^2+|p_i|^2}}\,dydp_i\cr&\le \frac{4C_1}{C_3^2}\int_{p_i^1>0}\frac{\omega_i}{p_i^1}\int_0^xe^{-\frac{\omega_i}{p_i^1}(x-y)}e^{-C_3|p_i^1|}\,dydp_i^1\\
&= \frac{4C_1}{C_3^2}\int_{p_i^1>0}\big(1-e^{-\frac{\omega_i}{p_i^1}x}\big)e^{-C_3|p_i^1|}\,dp_i^1.
\end{split}
\end{align}
We split the domain into three cases $p_i^1\in (0,\omega_i]$, $p_i^1\in (\omega_i,1/\omega_i]$, and $p_i^1\in (1/\omega_i,\infty)$ and estimate them separately:
\begin{align*}
	\int_{p_i^1>0}\big(1-e^{-\frac{\omega_i}{p_i^1}x}\big)e^{-C_3|p_i^1|}\,dp_i^1=\left\{\int_{p_i^1<\omega_i}+\int_{\omega_i\leq p_i^1<1/\omega_i}+\int_{1/\omega_i\leq p_i^1}\right\}\big(1-e^{-\frac{\omega_i}{p_i^1}x}\big)e^{-C_3|p_i^1|}\,dp_i^1.
\end{align*}
\noindent\newline
$\bullet$ (Case of $p_i^1\in (0,\omega_i]$): For the first case, we can easily obtain 
\begin{equation}\label{I1}
 \int_{0<p_i^1\le \omega_i} e^{-C_3|p_i^1|}\big(1-e^{-\frac{\omega_i}{p_i^1}x}\big)\,dp_i^1\le 	\int_{0<p_i^1\le \omega_i}\,dp_i^1\le \omega_i.
\end{equation}
\noindent\newline
$\bullet$ (Case of $p_i^1\in (\omega_i,1/\omega_i]$): We use the inequality
\begin{equation}\label{inequality}
1-e^{-x}\le x\quad \mbox{for all}\ x\in [0,\infty)
\end{equation}
to see that
\begin{align}\label{I2}
 \int_{\omega_i<p_i^1\le \frac{1}{\omega_i}} e^{-C_3|p_i^1|}\big(1-e^{-\frac{\omega_i}{p_i^1}x}\big)\,dp_i^1&\le \int_{\omega_i<p_i^1\le\frac{1}{\omega_i}} e^{-C_3|p_i^1|}\frac{\omega_i}{p_i^1}x\, dp_i^1\cr 
&\le \int_{\omega_i<p_i^1\le\frac{1}{\omega_i}}\frac{\omega_i}{p_i^1}dp_i^1\cr
&=2\omega_i\ln \frac{1}{\omega_i}.
\end{align}
$\bullet$ (Case of $p_i^1\in (1/\omega_i,\infty)$): It follows from \eqref{inequality} that 
\begin{align}\label{II}
\begin{split}
\int_{p_i^1>\frac{1}{\omega_i}}\big(1-e^{-\frac{\omega_i}{p_i^1}x}\big)e^{-C_3|p_i^1|}\,dp_i^1
& \le \int_{p_i^1>\frac{1}{\omega_i}} \frac{\omega_i}{p_i^1}x e^{-C_3|p_i^1|}\,dp_i^1\cr 
&\le \omega_i^2 \int_{p_i^1>\frac{1}{\omega_i}}e^{-C_3|p_i^1|}dp_i^1\cr
&= \frac{\omega_i^2}{C_3}e^{-\frac{C_3}{\omega_i}}.
\end{split}
\end{align}
Going back to \eqref{p1>0} with \eqref{I1}, \eqref{I2} and \eqref{II}, we get the desired result:
\begin{equation*}
\int_{p_i^1>0}\frac{\omega_i}{p_i^1}\int_0^xe^{-\frac{\omega_i}{p_i^1}(x-y)}\mathcal{J}_i(y)\,dydp_i\le C\left(\omega_i\ln\frac{1}{\omega_i}+\omega_i+\omega_i^2e^{-\frac{C}{\omega_i}}\right).
\end{equation*}
\end{proof}
\section{Solution operator mapping into itself}
Let us define the solution operator $\Psi=(\Psi_1,\cdots,\Psi_N)$ by
\begin{align*}
	\Psi_i(f)(x,p_i)&=\left(e^{-\frac{\omega_i}{|p_i^1|}x}f_{i,L}+\frac{\omega_i}{|p_i^1|}\int_0^xe^{-\frac{\omega_i}{|p_i^1|}(x-y)}\mathcal{J}_i(y)\, dy\right)\bold{1}_{p_i^1>0}\cr
	&+\left(e^{-\frac{\omega_i}{|p_i^1|}(1-x)}f_{i,R}+\frac{\omega_i}{|p_i^1|}\int_x^1e^{-\frac{\omega_i}{|p_i^1|}(y-x)}\mathcal{J}_i(y)\,dy\right)\bold{1}_{p_i^1<0}.
%	&=:\Psi^+(f_i)\bold{1}_{p_i^1>0}+\Psi^-(f_i)\bold{1}_{p_i^1<0}.
\end{align*}

The aim of this section is to show that the operator $\Psi$ maps the solution space $\Omega$ into itself, i.e., $\Psi(f)$ satisfies $(\mathcal{A})$ and $(\mathcal{B})$ for $f\in \Omega$. 
The proof consists of the following two lemmas.
\begin{lemma} Let $f=(f_1,\cdots,f_N)$ be an element of the solution space $\Omega$. Then, for sufficiently small $\omega_i$, we have
$$
\Psi_i(f)\ge 0,\quad \int_{\mathbb{R}^3} \Psi_i(f)\,\frac{dp_i}{p_i^0}\ge a_{i,\ell},\quad \int_{\mathbb{R}^3}\Psi_i(f)\,dp_i\le a_{i,u}
$$
for every $i\in \{1,\cdots,N\}$.
\end{lemma}
\begin{proof}
$\bullet$ (Non-negativity): By the second condition of $(\mathcal{A})$, we readily check 
\begin{align}\label{phi}
	\begin{split}
		\Psi_i(f)&=\left(e^{-\frac{\omega_i}{|p_i^1|}x}f_{i,L}+\frac{\omega_i}{|p_i^1|}\int_0^xe^{-\frac{\omega_i}{|p_i^1|}(x-y)}\mathcal{J}_i(y)\, dy\right)\bold{1}_{p_i^1>0}\cr
		&+\left(e^{-\frac{\omega_i}{|p_i^1|}(1-x)}f_{i,R}+\frac{\omega_i}{|p_i^1|}\int_x^1e^{-\frac{\omega_i}{|p_i^1|}(y-x)}\mathcal{J}_i(y)\,dy\right)\bold{1}_{p_i^1<0}\cr
		&\ge e^{-\frac{\omega_i}{|p_i^1|}x}f_{i,L}\bold{1}_{p_i^1>0}+e^{-\frac{\omega_i}{|p_i^1|}(1-x)}f_{i,R}\bold{1}_{p_i^1<0}\cr
		&\ge 0.
	\end{split}
\end{align} 
\noindent\newline
$\bullet$ (Lower bound): It follows from \eqref{phi} that
\begin{align}\label{Psi lower}\begin{split}
		\Psi_i(f)&\ge e^{-\frac{\omega_i}{|p_i^1|}x}f_{i,L}\bold{1}_{p_i^1>0}+e^{-\frac{\omega_i}{|p_i^1|}(1-x)}f_{i,R}\bold{1}_{p_i^1<0}\cr
		&\ge e^{-\frac{\omega_i}{|p_i^1|}}f_{i,LR},
\end{split}\end{align}
which leads to
$$
\int_{\mathbb{R}^3} \Psi_i(f)\,\frac{dp_i}{p_i^0}\ge \int_{\mathbb{R}^3} e^{-\frac{\omega_i}{|p_i^1|}}f_{i,LR}\,\frac{dp_i}{p_i^0}.
$$
For sufficiently small $r>0$, we have
\begin{align}\label{1/4}
	\int_{\mathbb{R}^3} e^{-\frac{\omega_i}{|p_i^1|}}f_{i,LR}\,\frac{dp_i}{p_i^0}\geq e^{-\frac{\omega_i}{r}}\int_{|p_i^1|>r}f_{i,LR}\,\frac{dp_i}{p_i^0}\geq \frac{e^{-\frac{\omega_i}{r}}}{2}\int_{\mathbb{R}^3}f_{i,LR}\,\frac{dp_i}{p_i^0}.
\end{align}
Thus, choosing $\omega_i$ small enough to satisfy $e^{-\frac{\omega_i}{r}}>\frac{1}{2}$, we get
\begin{align}\label{1/4-2}
	\int_{\mathbb{R}^3} \Psi_i(f)\,\frac{dp_i}{p_i^0}\geq \frac{1}{4}\int_{\mathbb{R}^3}f_{i,LR}\,\frac{dp_i}{p_i^0}=a_{i,\ell}.
\end{align}
$\bullet$ (Upper bound): Recall from Lemma \ref{phi2} that
\begin{align*} \begin{split}
		\int_{\mathbb{R}^3}\Psi_i(f)\,dp_i&=\int_{\mathbb{R}^3} e^{-\frac{\omega_i}{|p_i^1|}x}f_{i,L}\bold{1}_{p_i^1>0}+e^{-\frac{\omega_i}{|p_i^1|}(1-x)}f_{i,R}\bold{1}_{p_i^1<0}\,dp_i\cr
		&+ \int_{p_i^1>0}\frac{\omega_i}{|p_i^1|}\int_0^xe^{-\frac{\omega_i}{|p_i^1|}(x-y)}\mathcal{J}_i(y)\, dydp_i+\int_{p_i^1<0}\frac{\omega_i}{|p_i^1|}\int_x^1e^{-\frac{\omega_i}{|p_i^1|}(y-x)}\mathcal{J}_i(y)\,dydp_i\cr
		&\le\int_{\mathbb{R}^3} f_{i,LR}\,dp_i+C\left(\omega_i\ln\frac{1}{\omega_i}+\omega_i+\omega_i^2e^{-\frac{C}{\omega_i}}\right).
\end{split}\end{align*}
We then choose a sufficiently small $\omega_i$ again so that
$$
C\left(\omega_i\ln\frac{1}{\omega_i}+\omega_i+\omega_i^2e^{-\frac{C}{\omega_i}}\right)\le \int_{\mathbb{R}^3} f_{i,LR} \,dp_i
$$
which gives
$$
\int_{\mathbb{R}^3} \Psi_i(f)\,dp_i\le 2\int_{\mathbb{R}^3} f_{i,LR} dp_i=a_{i,u}.
$$

\end{proof}
In the following lemma, to avoid confusion, we will denote the momentum moments of $\Psi(f)$ by $n^\Psi_i$ and $U^\Psi_i$:
\begin{align*}
		U_i^{\Psi\mu}&=\frac{c}{n_i}\int_{\mathbb{R}^3}p^\mu_i\Psi_i(f) \,\frac{dp_i}{p_i^0},\\
		n^\Psi_i&=\frac{1}{c}\int_{\mathbb{R}^3}p^\mu_i U^\Psi_{i \mu }\Psi_i(f) \,\frac{dp_i}{p_i^0},
 \end{align*}
and $\widetilde{\beta}^\Psi$ is also determined by the following relation: 
\begin{align}\label{beta into}
	\sum_{i=1}^{N}\frac{m_i}{\tau_i}\Phi(\widetilde{\beta}^\Psi)\int_{\mathbb{R}^3}\Psi_i(f) \,\frac{dp_i}{p_i^0}=\frac{1}{c}\left[\left( \sum_{i=1}^{N}\frac{m_i }{\tau_i}n^\Psi_iU_i^{\Psi\mu}\right)\left( \sum_{j=1}^{N}\frac{m_j }{\tau_j}n^\Psi_jU^\Psi_{j\mu}\right)\right]^{\frac{1}{2}}.
\end{align}
%  Then, $\widetilde\beta$ is understood as a function of $\Psi(f)=(\Psi_1(f),\cdots,\Phi_N(f))$ defined through the relation  \eqref{beta identity}
\begin{lemma}\label{222} Let $f=(f_1,\cdots,f_N)$ be an element of the solution space $\Omega$. Then, for sufficiently small $\omega_i$, $\Psi(f)=(\Psi_1(f),\cdots,\Psi_N(f))$ satisfies
	$$
	\gamma\sum_{i=1}^{N}cm_i\omega_i \lambda_i  \int_{\mathbb{R}^3} f_{i,LR}\,\frac{dp_i}{p_i^0}\le \sum_{i=1}^{N}\omega_i \Phi_i (\widetilde{\beta}^\Psi ) \sqrt{\lambda_i} \int_{\mathbb{R}^3} f_{i,LR}\,\frac{dp_i}{p_i^0}.
	$$
	
\end{lemma}
\begin{proof}
	% Recall from \eqref{beta identity} that $\widetilde{\beta}^\Psi$ is determined by the relation 
	%\begin{equation}\label{beta into}
	%\sum_{i=1}^{N}\omega_i \Phi_i(\widetilde{\beta}^\Psi) \int_{\mathbb{R}^3}\Psi_i(f) \,\frac{dp_i}{p_i^0}=\frac{1}{c}\left[\left( \sum_{i=1}^{N}\frac{m_i }{\tau_i}n_i^\PsiU_i^{\Psi\mu}\right)\left( \sum_{j=1}^{N}\frac{m_j }{\tau_j}n_jU_{j\mu}^\Psi\right)\right]^{\frac{1}{2}}.
	%\end{equation}
	To investigate the lower bound of the right-hand side of \eqref{beta into}, we observe that
	$$
	U_i^{\Psi\mu} U^\Psi_{j\mu}=\sqrt{c^2+|U_i^\Psi|^2}\sqrt{c^2+|U_j^\Psi|^2}-U_i^\Psi\cdot U_j^\Psi \ge c^2
	$$
	due to the Cauchy-Schwarz inequality. Using this, one finds 
	\begin{align}\label{RHS} \begin{split}
			\left[\left( \sum_{i=1}^{N}\omega_in_i^\Psi U_i^{\Psi\mu}\right)\left( \sum_{j=1}^{N}\omega_jn_j^\Psi U_{j\mu}^\Psi\right)\right]^{\frac{1}{2}}	&= \left( \sum_{i=1}^{N}\omega_i \omega_jn_i^\Psi n_j^\Psi U_i^{\Psi\mu} U^\Psi_{j\mu}\right)^{\frac{1}{2}}\cr
			&\ge  c \sum_{i=1}^{N}\omega_in_i^\Psi \cr
			&=c\sum_{i=1}^{N}\omega_i \left\{  \left(\int_{\mathbb{R}^3} \Psi_i(f)\,dp_i\right)^2-\sum_{j=1}^3\left(\int_{\mathbb{R}^3} p_i^j\Psi_i(f)\,\frac{dp_i}{p_i^0}\right)^2  \right\}^{\frac{1}{2}}. 
	\end{split}\end{align}
	In the last line, we used the definition of $n_i^\Psi$. Applying H\"{o}lder's inequality, we obtain
	\begin{align}\label{Holder}\begin{split}
			&\left(\int_{\mathbb{R}^3} \Psi_i(f)\,dp_i\right)^2-\sum_{j=1}^3\left(\int_{\mathbb{R}^3} p_i^j\Psi_i(f)\,\frac{dp_i}{p_i^0}\right)^2\cr
			&\ge \left(\int_{\mathbb{R}^3} \Psi_i(f)\,dp_i\right)^2- \int_{\mathbb{R}^3}  \Psi_i(f)\, dp_i\int_{\mathbb{R}^3} \frac{|p_i|^2}{(cm_i)^2+|p_i|^2}\Psi_i(f)\,dp_i  \cr
			&=\int_{\mathbb{R}^3}  \Psi_i(f)\, dp_i\int_{\mathbb{R}^3} \frac{(cm_i)^2}{(cm_i)^2+|p_i|^2}\Psi_i(f)\,dp_i,
	\end{split}\end{align}
	where we recall $(p_i^0)^2=(cm_i)^2+|p_i|^2$.
	As in \eqref{1/4} and \eqref{1/4-2}, we can take $\omega_i$ small enough so that
	\begin{align}\label{gamma1}
		\int_{\mathbb{R}^3}  \Psi_i(f)\, dp_i\int_{\mathbb{R}^3} \frac{(cm_i)^2}{(cm_i)^2+|p_i|^2}\Psi_i(f)\,dp_i
		\geq \gamma(cm_i)^2 \int_{\mathbb{R}^3}  f_{i,LR}\, dp_i\int_{\mathbb{R}^3} \frac{1}{(cm_i)^2+|p_i|^2}f_{i,LR}\,dp_i. 
	\end{align}
	Combining \eqref{Psi lower}, \eqref{RHS}, \eqref{Holder}, and \eqref{gamma1}, we get
	\begin{align*}
		&\frac{1}{c}\left[\left( \sum_{i=1}^{N}\omega_in_i^\Psi U_i^{\Psi\mu}\right)\left( \sum_{j=1}^{N}\omega_jn_j^\Psi U_{j\mu}^\Psi\right)\right]^{\frac{1}{2}}\ge   \gamma\sum_{i=1}^{N}cm_i\omega_i \left\{ \int_{\mathbb{R}^3}  f_{i,LR}\, dp_i\int_{\mathbb{R}^3} \frac{1}{(p_i^0)^2}f_{i,LR}\,dp_i  \right\}^{\frac{1}{2}} 
	\end{align*}	
	which, together with the definition of $\lambda_i$ (see \eqref{lambda}), gives
	\begin{align}\label{beta into3} \begin{split}
			&\frac{1}{c}\left[\left( \sum_{i=1}^{N}\omega_in_i^\Psi U_i^{\Psi\mu}\right)\left( \sum_{j=1}^{N}\omega_jn_j^\Psi U_{j\mu}^\Psi\right)\right]^{\frac{1}{2}}\ge   \gamma\sum_{i=1}^{N}cm_i\omega_i  \lambda_i   \int_{\mathbb{R}^3} f_{i,LR}\,\frac{dp_i}{p_i^0}.   
	\end{split}\end{align}
	On the other hand, it follows from Lemma \ref{phi2} that 
	\begin{align}\label{beta into2}\begin{split}
			\int_{\mathbb{R}^3}\Psi_i(f)\,\frac{dp_i}{p_i^0}&=\int_{\mathbb{R}^3} e^{-\frac{\omega_i}{|p_i^1|}x}f_{i,L}\bold{1}_{p_i^1>0}+e^{-\frac{\omega_i}{|p_i^1|}(1-x)}f_{i,R}\bold{1}_{p_i^1<0}\,\frac{dp_i}{p_i^0}\cr
			&+ \int_{p_i^1>0}\frac{\omega_i}{|p_i^1|}\int_0^xe^{-\frac{\omega_i}{|p_i^1|}(x-y)}\mathcal{J}_i(y)\, dy\frac{dp_i}{p_i^0}+\int_{p_i^1<0}\frac{\omega_i}{|p_i^1|}\int_x^1e^{-\frac{\omega_i}{|p_i^1|}(y-x)}\mathcal{J}_i(y)\,dy\frac{dp_i}{p_i^0}\cr
			&\le\int_{\mathbb{R}^3} f_{i,LR}\,\frac{dp_i}{p_i^0}+C\left(\omega_i\ln\frac{1}{\omega_i}+\omega_i+\omega_i^2e^{-\frac{C}{\omega_i}}\right).
	\end{split}\end{align}
	Going back to \eqref{beta into} with \eqref{beta into2} and \eqref{beta into3}, we have
	\begin{align*}
		&\sum_{i=1}^{N}\omega_i \Phi_i(\widetilde{\beta}^\Psi) \left\{ \int_{\mathbb{R}^3} f_{i,LR}\,\frac{dp_i}{p_i^0}+C\left(\omega_i\ln\frac{1}{\omega_i}+\omega_i+\omega_i^2e^{-\frac{C}{\omega_i}}\right) \right\}\cr
		&\ge \gamma\sum_{i=1}^{N}cm_i\omega_i \lambda_i   \int_{\mathbb{R}^3} f_{i,LR}\,\frac{dp_i}{p_i^0}.
	\end{align*}
	We choose sufficiently small $\omega_i$ satisfying
	$$
	C\left(\omega_i\ln\frac{1}{\omega_i}+\omega_i+\omega_i^2e^{-\frac{C}{\omega_i}}\right)\le  \left(\sqrt{\lambda_i}-1\right) \int_{\mathbb{R}^3} f_{i,LR}\,\frac{dp_i}{p_i^0}
	$$
	where $\lambda_i$ is strictly bigger than $1$  for all $x\in [0,1]$ (see Remark \ref{lambda1}). Thus we have
	\begin{align*}
		\gamma\sum_{i=1}^{N}cm_i\omega_i \lambda_i  \int_{\mathbb{R}^3} f_{i,LR}\,\frac{dp_i}{p_i^0}\le \sum_{i=1}^{N}\omega_i \Phi_i(\widetilde{\beta}^\Psi) \sqrt{\lambda_i} \int_{\mathbb{R}^3} f_{i,LR}\,\frac{dp_i}{p_i^0}.
	\end{align*}

\end{proof}

\section{Contraction mapping}
This section is devoted to showing the Lipschitz continuity of the solution operator $\Psi$ in the solution space $\Omega$, which completes the proof of Theorem \ref{main} due to the Banach fixed-point theorem. For this, we first estimate auxiliary parameters consisting of $\mathcal{J}_i$.
\begin{lemma}\label{lip}
Let $f=(f_1,\cdots,f_N)$ and $g=(g_1,\cdots,g_N)$ be elements of the solution space $\Omega$. We then have
	\begin{align*}
		|\widetilde{U}_f-\widetilde{U}_g|+ |\widetilde{\beta}_f-\widetilde{\beta}_g|\leq C\sum_{i=1}^N\|f_i-g_i\|_{L^1(\mathbb{R}^3)}.
	\end{align*}
\begin{proof}
$\bullet$ (Estimate of $\widetilde U$):	For notational convenience, we define 
	\begin{align}\label{A}
		A_f^\mu:=\sum_{i=1}^N\omega_in_{f_i}U^\mu_{f_i}.
	\end{align}
Then we have 
\begin{align}\label{Ulip}
	|\widetilde{U}_f-\widetilde{U}_g|&= c\frac{\left|\sqrt{A_f^\mu A_{f\mu}}A_g-\sqrt{A_g^\mu A_{g\mu}}A_f\right|}{\sqrt{A_f^\mu A_{f\mu}A_g^\mu A_{g\mu}}}.
\end{align}
For the numerator, we get
\begin{align}\label{numerator}
	\begin{split}
	&\left|\sqrt{A_f^\mu A_{f\mu}}A_g-\sqrt{A_g^\mu A_{g\mu}}A_f\right|\\
	&\leq \sqrt{A_f^\mu A_{f\mu}}\left|A_g-A_f\right|+|A_f|\left|\sqrt{A_f^\mu A_{f\mu}}-\sqrt{A_g^\mu A_{g\mu}}\right|\\
	&\leq \sqrt{A_f^\mu A_{f\mu}}\left|A_g-A_f\right|+|A_f|\frac{|A_f^\mu(A_{f\mu}-A_{g\mu})|+|A_g^\mu(A_{f\mu}-A_{g\mu})|}{\sqrt{A_f^\mu A_{f\mu}}+\sqrt{A_g^\mu A_{g\mu}}}\\
	&\leq \sqrt{A_f^\mu A_{f\mu}}\left|A_g-A_f\right|+|A_f|\frac{(A_f^0+A_g^0)|A_{f0}-A_{g0}|+(|A_f|+|A_g|)|A_f-A_g|}{\sqrt{A_f^\mu A_{f\mu}}+\sqrt{A_g^\mu A_{g\mu}}}.
\end{split}
\end{align}
It follows from \eqref{lowerbound} and \eqref{upperbound} that
\begin{align}\label{AA}
	\begin{split}
	\left(c\sum_{i=1}^{N}\omega_im_i a_{i,\ell}\right)^2\leq\left(c\sum_{i=1}^{N}\omega_i m_i \int_{\mathbb{R}^3}  f_i  \,\frac{dp_i}{p_i^0} \right)^2&\leq A_f^\mu A_{f\mu} \leq\left(2c\sum_{i=1}^{N}\omega_i \int_{\mathbb{R}^3}  f_i  dp_i\right)^2\leq \left(2c\sum_{i=1}^{N}\omega_i a_{i,u}\right)^2.\end{split}
\end{align}
Also, due to the property $(\mathcal{A})$, we have
\begin{align}\label{AAA}
	|A_f|+A_f^0\le c\sum_{i=1}^N\omega_i\left(\int_{\mathbb{R}^3}|p_i|f_i\,\frac{dp_i}{p_i^0}+\int_{\mathbb{R}^3}f_i\, dp_i\right)  \leq 2c\sum_{i=1}^N\omega_i \int_{\mathbb{R}^3}f_i\,dp_i\leq  2c\sum_{i=1}^N\omega_i a_{i,u}.
\end{align}
Obviously, these observations hold true for $A_g$.  Thus, these observations \eqref{Ulip}, \eqref{numerator}, \eqref{AA}, and \eqref{AAA} give
\begin{align*}
	|\widetilde{U}_f-\widetilde{U}_g|&\leq C|A_f-A_g|\\
	&\leq C\left|\sum_{i=1}^N\omega_ic\int_{\mathbb{R}^3}f_i-g_i\frac{dp_i}{p_i^0}\right|\\
	&\leq C\sum_{i=1}^N\|f_i-g_i\|_{L^1(\mathbb{R}^3)}.
\end{align*}
$\bullet$ (Estimate of $\widetilde\beta$): We recall that $\widetilde{\beta}_f$ and $\widetilde{\beta}_g$ satisfy
\begin{align*}
	\sum_{i=1}^{N}\omega_i\Phi_i(\widetilde{\beta}_f)\int_{\mathbb{R}^3}f_i \ \frac{dp_i}{p_i^0}&=\frac{1}{c}\left[\left( \sum_{i=1}^{N}\omega_in_{f_i}U_{f_i}^\mu\right)\left( \sum_{j=1}^{N}\omega_jn_{f_j}U_{{f_j}\mu}\right)\right]^{\frac{1}{2}}\\
	\sum_{i=1}^{N}\frac{m_i}{\tau_i}\Phi_i(\widetilde{\beta}_g)\int_{\mathbb{R}^3}g_i \ \frac{dp_i}{p_i^0}&=\frac{1}{c}\left[\left( \sum_{i=1}^{N}\omega_in_{g_i}U_{g_i}^\mu\right)\left( \sum_{j=1}^{N}\omega_jn_{g_j}U_{{g_j}\mu}\right)\right]^{\frac{1}{2}}.
\end{align*}
From this and the notation \eqref{A}, we can obtain
\begin{align}\label{tri}
	\begin{split}
	c\sum_{i=1}^N\omega_i\left(\Phi_i(\widetilde{\beta}_f)-\Phi_i(\widetilde{\beta}_g)\right)\int_{\mathbb{R}^{3}}f_i\frac{dp_i}{p_i^0}=& \ c\sum_{i=1}^N\omega_i\Phi_i(\widetilde{\beta}_g)\left(\int_{\mathbb{R}^{3}}g_i\frac{dp_i}{p_i^0}-\int_{\mathbb{R}^{3}}f_i\frac{dp_i}{p_i^0}\right)\\
	&+\sqrt{A_f^\mu A_{f\mu}}-\sqrt{A_g^\mu A_{g\mu}}.
	\end{split}
\end{align}
We recall that each $\Phi_i$ is monotone decreasing so that all the signs of $\Phi_i(\widetilde{\beta}_f)-\Phi_i(\widetilde{\beta}_g)$, $(i=1,\cdots,N)$, are same. This, together with \eqref{tri}, gives that
\begin{align}\label{ineq111}
	\begin{split}
	&c\sum_{i=1}^N\omega_i\left|\Phi_i(\widetilde{\beta}_f)-\Phi_i(\widetilde{\beta}_g)\right|\int_{\mathbb{R}^{3}}f_i\frac{dp_i}{p_i^0}\cr &=\left|c\sum_{i=1}^N\omega_i\left(\Phi_i(\widetilde{\beta}_f)-\Phi_i(\widetilde{\beta}_g)\right)\int_{\mathbb{R}^{3}}f_i\frac{dp_i}{p_i^0}\right|\\
	&\leq c\sum_{i=1}^N\omega_i\Phi_i(\widetilde{\beta}_g)\left|\int_{\mathbb{R}^{3}}f_i\frac{dp_i}{p_i^0}-\int_{\mathbb{R}^{3}}g_i\frac{dp_i}{p_i^0}\right|+\left|\sqrt{A_f^\mu A_{f\mu}}-\sqrt{A_g^\mu A_{g\mu}}\right|.
%	&\leq c\sum_{i=1}^N\omega_i\Phi_i(\widetilde{\beta}_\ell)\|f_i-g_i\|_{L^1_p}+C\sum_{i=1}^N\omega_i\|f_i-g_i\|_{L^1_p}.
	\end{split}
\end{align}
Here we used Lemma \ref{11} to obtain
\begin{align*}
	c\sum_{i=1}^N\omega_i\Phi_i(\widetilde{\beta}_g)\left|\int_{\mathbb{R}^{3}}f_i\frac{dp_i}{p_i^0}-\int_{\mathbb{R}^{3}}g_i\frac{dp_i}{p_i^0}\right|&\leq c\sum_{i=1}^N\omega_i\Phi_i(\widetilde{\beta}_\ell)\|f_i-g_i\|_{L^1_p}\\
	&\leq C\sum_{i=1}^N\omega_i\|f_i-g_i\|_{L^1_p}
\end{align*}
where $C=c\max_{i=1,\cdots, N}\Phi_i(\widetilde{\beta}_\ell)>0$ in the last line. Also, for the second term of the last line \eqref{ineq111}, it follows from \eqref{numerator}, \eqref{AA}, and \eqref{AAA} that
\begin{align*}
	\left|\sqrt{A_f^\mu A_{f\mu}}-\sqrt{A_g^\mu A_{g\mu}}\right|&\leq \frac{(A_f^0+A_g^0)|A_{f0}-A_{g0}|+(|A_f|+|A_g|)|A_f-A_g|}{\sqrt{A_f^\mu A_{f\mu}}+\sqrt{A_g^\mu A_{g\mu}}}\\
	&\leq C |A_f-A_g|\\
	&= C\left|\sum_{i=1}^N\omega_ic\int_{\mathbb{R}^3}p_i(f_i-g_i)\frac{dp_i}{p_i^0}\right|\\
	&\leq C\sum_{i=1}^N\omega_i\|f_i-g_i\|_{L^1_p}.
\end{align*}
Combining these estimates with the property ($\mathcal{A}$), we have
\begin{align*}
	c\left(\min_{i=1,\cdots, N}a_{i,\ell}\right)\sum_{i=1}^N\omega_i\left|\Phi_i(\widetilde{\beta}_f)-\Phi_i(\widetilde{\beta}_g)\right|\leq C\sum_{i=1}^N\omega_i\|f_i-g_i\|_{L^1_p},
\end{align*}
which implies that for at least one $j\in \{1,\cdots,N\} $, 
$$
\left|\Phi_j(\widetilde{\beta}_f)-\Phi_j(\widetilde{\beta}_g)\right|\leq C\|f_j-g_j\|_{L^1_p}.
$$
By the mean value theorem, there exists a positive value $\theta\in[\widetilde{\beta}_\ell,\widetilde{\beta}_u]$ such that
$$
\left|\frac{\partial\Phi_j}{\partial\widetilde{\beta}}(\theta)\times(\widetilde{\beta}_f-\widetilde{\beta}_g)\right|\leq C\|f_j-g_j\|_{L^1_p}.
$$
We note that
$$
\min_{\theta\in[\widetilde{\beta}_\ell,\widetilde{\beta}_u]}\left|\frac{\partial\Phi_j}{\partial\widetilde{\beta}}(\theta)\right|>0,
$$
which concludes that
$$
|\widetilde{\beta}_f-\widetilde{\beta}_g|\leq C\|f_j-g_j\|_{L^1_p} \leq C\sum_{i=1}^N\|f_i-g_i\|_{L^1_p}.
$$
\end{proof}
\end{lemma}

\begin{lemma}\label{JJ} Let $f=(f_1,\cdots,f_N)$ and $g=(g_1,\cdots,g_N)$ be elements of the solution space $\Omega$. For sufficiently small $w_i$, we have
	$$
	|\mathcal{J}_i(f)-\mathcal{J}_i(g)|\le Ce^{-Cp_i^0}\sum_{i=1}^N\|f_i-g_i\|_{L^1(\mathbb{R}^3)}.
	$$
	\begin{proof}
		For brevity, we rewrite $$
		\mathcal{J}_i(f)\equiv J_i(\alpha,\widetilde{\beta},\widetilde{U})=\frac{\alpha}{\int_{\mathbb{R}^3} e^{-c\widetilde{\beta} p_i^0} \frac{dp_i}{p_i^0}}e^{-\widetilde{\beta}\widetilde{U}^\mu p_{i \mu }}.
		 $$
By the mean value theorem, there exists a constant $\theta\in [0,1]$ such that
\begin{align*}
	|\mathcal{J}_i(f)-\mathcal{J}_{i}(g)|&=|J_i(\alpha_{f_i},\widetilde{\beta}_f,\widetilde{U}_f)-J(\alpha_{g_i},\widetilde{\beta}_g,\widetilde{U}_g)|\\
	&=|\nabla_{\alpha,\widetilde{\beta},\widetilde{U}}J_i(\theta)\cdot(\alpha_{f_i}-\alpha_{g_i},\widetilde{\beta}_f-\widetilde{\beta}_g,\widetilde{U}_f-\widetilde{U}_g)|
\end{align*}
where $J_i(\theta)$ denotes
\begin{align*}
	J_i(\theta)=J_i\big((1-\theta)\alpha_{f_i}+\theta \alpha_{g_i},(1-\theta)\widetilde{\beta}_f+\theta \widetilde{\beta}_g,(1-\theta)\widetilde{U}_f+\theta\widetilde{U}_g\big),
\end{align*}
and $\alpha_{f_i},\alpha_{g_i}$ are given by
\begin{align*}
	\alpha_{f_i}=\int_{\mathbb{R}^3}f_i\frac{dp_i}{p_i^0},\qquad \text{and}\qquad \alpha_{g_i}=\int_{\mathbb{R}^3}g_i\frac{dp_i}{p_i^0}.
\end{align*}
Then, we have from Lemma \ref{lip} that
\begin{align*}
		|\mathcal{J}_{i}(f)-\mathcal{J}_{i}(g)|\leq C|\nabla_{\alpha,\widetilde{\beta},\widetilde{U}}J_i(\theta)|\sum_{i=1}^N\|f_i-g_i\|_{L^1(\mathbb{R}^3)}.
\end{align*}
By simple computations, we find
\begin{align*}
	\frac{\partial J_i}{\partial\alpha}=\frac{1}{\alpha}J_i,\quad  \frac{\partial J_i}{\partial\widetilde\beta}=\left(c\Phi_i(\widetilde{\beta})-\widetilde{U}^\mu p_{i\mu}\right)J_i,\quad \frac{\partial J_i}{\partial \widetilde{U}}=\frac{\widetilde{\beta}}{k}\left(p_i-\frac{\widetilde{U}p_i^0}{\widetilde{U}^0}\right)J_i,
\end{align*}
where we used the fact that
\begin{align*}
	\frac{\partial}{\partial\widetilde{\beta}}\left(\frac{1}{\int_{\mathbb{R}^3} e^{-c\widetilde{\beta} p_i^0} \frac{dp_i}{p_i^0}}\right)=\frac{c\int_{\mathbb{R}^3} e^{-c\widetilde{\beta} p_i^0} dp_i}{\left(\int_{\mathbb{R}^3} e^{-c\widetilde{\beta} p_i^0} \frac{dp_i}{p_i^0}\right)^2}=\frac{c\Phi_i(\widetilde{\beta})}{\int_{\mathbb{R}^3} e^{-c\widetilde{\beta} p_i^0} \frac{dp_i}{p_i^0}}.
\end{align*}
Lemma \ref{11} and Lemma \ref{22}, together with $\alpha_{f_i},\alpha_{g_i}\geq a_{i,\ell}$, give
\begin{align*}
	\left|\frac{\partial J_i}{\partial\alpha}\right|&\leq Ce^{-Cp_i^0},\\
	\left|\frac{\partial J_i}{\partial\beta}\right|&\leq  c\Phi_i(\widetilde{\beta}_\ell)J_i+2\widetilde{U}_up_i^0J_i\leq  Ce^{-Cp_i^0},\\
	\left|\frac{\partial J_i}{\partial\widetilde{U}}\right|&\leq 2\frac{\widetilde{\beta}_u}{k}p_i^0J_i\leq Ce^{-Cp_i^0},
\end{align*}
where we used the elementary inequality $xe^{-x}\leq Ce^{-Cx}.$
\end{proof}

\end{lemma}

\begin{proposition} Let $f=(f_1,\cdots,f_N)$ and $g=(g_1,\cdots,g_N)$ be elements of the solution space $\Omega$. For sufficiently small $\omega_i$, there exists a constant $0<\delta<1$ such that
\[
\|\Psi_i(f)-\Psi_i(g)\|\leq\delta \sum_{i=1}^N\|f_i-g_i\| .
\]
\end{proposition}
\begin{proof}
By definition of $\Psi$, we see that
\begin{align*}
&\int_{\mathbb{R}^3} |\Psi_i(f)-\Psi_i(g)|dp_i\cr
&=\int_{p_1>0}|\Psi_i(f)-\Psi_i(g)|dp_i+\int_{p_1<0}|\Psi_i(f)-\Psi_i(g)|dp_i \cr
%&\qquad\le\int_{q_1>0}e^{-\frac{w}{q_{1}}x}|f_L-g_L|+\frac{w}{q_1}\int_0^xe^{-\frac{w}{q_1}(x-y)}|J_f-J_g|dydq\cr
%&\qquad+\int_{q_1<0}e^{-\frac{w}{|q_{1}|}(1-x)}|f_R-g_R|+\frac{w}{|q_1|}\int_x^1e^{-\frac{w}{|q_1|}(y-x)}|J_f-J_g|dydq \cr
&\leq \int_{p_1>0}\int_0^x\frac{\omega_i}{p_i^1}e^{-\frac{\omega_i}{p_i^1}(x-y)}|\mathcal{J}_i(f)-\mathcal{J}_i(g)|dydp_i+ \int_{p_i^1<0}\int_x^1 \frac{\omega_i}{|p_i^1|}e^{-\frac{\omega_i}{|p_i^1|}(y-x)}|\mathcal{J}_i(f)-\mathcal{J}_i(g)|dydp_i.
\end{align*}
Applying Lemma \ref{JJ}, we obtain
\begin{align*}
&\int_{p_1>0}\int_0^x\frac{\omega_i}{p_i^1}e^{-\frac{\omega_i}{p_i^1}(x-y)}|\mathcal{J}_i(f)-\mathcal{J}_i(g)|dydp_i+\int_{p_i^1<0}\int_x^1 \frac{\omega_i}{|p_i^1|}e^{-\frac{\omega_i}{|p_i^1|}(y-x)}|\mathcal{J}_i(f)-\mathcal{J}_i(g)|dydp_i\\
&\le C\biggl\{\int_{p_i^1>0}\int_0^x\frac{\omega_i}{p_i^1} e^{-\frac{\omega_i}{p_i^1}(x-y)}e^{-Cp_i^0}dydp_i+\int_{p_i^1<0}\int_x^1 \frac{\omega_i}{|p_i^1|}e^{-\frac{\omega_i}{|p_i^1|}(y-x)}e^{-Cp_i^0}dydp_i\biggl\}\sum_{i=1}^N\|f_i-g_i\|_{L^1(\mathbb{R}^3)}.
\end{align*}
In the almost same manner as in the proof of Lemma \ref{phi2}, we can further estimate the last term of the above as follows: 
\begin{align*}
	 &\biggl\{\int_{p_i^1>0}\int_0^x\frac{\omega_i}{p_i^1} e^{-\frac{\omega_i}{p_i^1}(x-y)}e^{-Cp_i^0}dydp_i+\int_{p_i^1<0}\int_x^1 \frac{\omega_i}{|p_i^1|}e^{-\frac{\omega_i}{|p_i^1|}(y-x)}e^{-Cp_i^0}dydp_i\biggl\}\sum_{i=1}^N\|f_i-g_i\|_{L^1(\mathbb{R}^3)}\cr
	&\le C\left(\omega_i\ln\frac{1}{\omega_i}+\omega_i+\omega_i^2e^{-\frac{C}{\omega_i}}\right)\sum_{i=1}^N\|f_i-g_i\|_{L^1(\mathbb{R}^3)},
\end{align*} 
which gives
$$
\|\Psi_i(f)-\Psi_i(g)\|_{L^1(\mathbb{R}^3)}\le C\left(\omega_i\ln\frac{1}{\omega_i}+\omega_i+\omega_i^2e^{-\frac{C}{\omega_i}}\right)\sum_{i=1}^N\|f_i-g_i\|_{L^1(\mathbb{R}^3)}.
$$
Therefore, integrating in $x\in [0,1]$, adding over $i=1,\cdots, N$, and letting $w_i$ sufficiently small, we obtain the desired results.
\end{proof}
\noindent{\bf Acknowledgement}
B.-H. Hwang was supported by Basic Science Research Program through the National Research Foundation of Korea(NRF) funded by the Ministry of Education (No. 2019R1A6A1A10073079). Myeong-Su Lee was supported by Basic Science Research Program through the National Research Foundation of Korea (NRF) funded by the Ministry of Education (RS-2023-00244475).

\end{document}